\documentclass[reqno,11pt]{article}
\usepackage{amsmath,amsthm,amssymb,amsfonts}
\usepackage{color}

\newtheorem{theorem}{Theorem}[section]
\newtheorem{corollary}[theorem]{Corollary}
\newtheorem{lemma}[theorem]{Lemma}
\newtheorem{conj}[theorem]{Conjecture}
\newtheorem{remark}[theorem]{Remark}

\newtheorem{proposition}[theorem]{Proposition}

\def\a{\alpha}
\def\v{\widetilde v_\lambda}
\def\R{\mathbb R}
\def\e{\varepsilon}
\def\c{{\rm Cap}}

\begin{document}

\title{Brunn-Minkowski inequality for the $1$-Riesz capacity \\ 
and level set convexity for the $1/2$-Laplacian}
 
 \author{M. Novaga
 \thanks{Dipartmento di Matematica,
 Universit\`a di Pisa, Largo Bruno Pontecorvo 5, 56127 Pisa, Italy 
 \newline email: novaga@dm.unipi.it},
 B. Ruffini
 \thanks{Scuola Normale Superiore, Piazza dei Cavalieri 7, 56126 Pisa, Italy
 \newline email: berardo.ruffini@sns.it}}


\maketitle

\begin{abstract}
\noindent We prove that that the $1$-Riesz capacity satisfies a Brunn-Minkowski inequality,
and that the capacitary function of the $1/2$-Laplacian is level set convex.
\end{abstract}


%
\section{Introduction}

In this paper we consider the following problem
\begin{equation}\label{problemaprincipale}
\left\{ \begin{array}{ll}
 (-\Delta)^{s}u=0\qquad &\mbox{on}\,\,\,\R^N\setminus K\\
 u=1\qquad &\mbox{on}\,\,\,K\\
 \lim_{|x|\to+\infty}u(x)=0
 \end{array}\right.
\end{equation}
where $N\ge 2$, $s\in (0,N/2)$, and $(-\Delta)^{s}$ stands for the 
$s$-fractional Laplacian, defined as the unique pseudo-differential operator $(-\Delta)^s:\mathcal S\mapsto
L^2(\R^N)$, being $\mathcal{S}$ the Schwartz space of functions with fast decay to $0$ at infinity, such that 
\[
 \mathcal F(-\Delta)^{s}f=|\xi|^{2s}\mathcal F (f)(\xi),
\]
where $\mathcal{F}$ denotes the Fourier transform.
We refer to the guide \cite[Section $3$]{dinpalval} for more details on the subject.
A quantity strictly related to Problem \eqref{problemaprincipale} is the so-called {\em Riesz potential energy} 
of a set $E$, defined as
\begin{equation}\label{rieszpotentialenergy}
 I_\a(E)=\inf_{\mu(E)=1}\int_{\R^N\times\R^N}\frac{d\mu(x)\,d\mu(y)}{|x-y|^{N-\a}}
\qquad \a\in (0,N). 
\end{equation}
It is possible to prove (see \cite{landkof}) that  if $E$ is a compact set, then 
the infimum in the
definition of $\mathcal I_\alpha(E)$
is achieved by a Radon measure $\mu$  supported on the boundary of $E$ if $\alpha\le N-2$, and with support equal
to the whole $E$ if $\alpha\in(N-2,N)$. If $\mu$ is the optimal measure for the set $E$, we define
the {\em Riesz potential} of $E$ as
\begin{equation}\label{rieszpotential}
 v(x)=\int_{\R^N}\frac{d\mu(y)}{|x-y|^{N-\a}},
\end{equation}
so that
\[ 
I_\a(E)=\int_{\R^N}v(x)d\mu(x).
\]
It is not difficult to check (see \cite{landkof,golnovruf}) that the potential $v$ satisfies 
\[
(-\Delta)^{\frac{\a}{2}}v=c(\a,N)\,\mu,
\]
where $c(\a,N)$ is a positive constant, and that $v=I_\a(E)$ on $E$. 
In particular, if $s=\alpha/2$, then $v_K=v/I_{2s}(K)$ is the unique 
solution of Problem \eqref{problemaprincipale}.

Following \cite{landkof}, we define the {\em $\alpha$-Riesz capacity} of a set $E$ as
\begin{equation}\label{rieszcapacity}
 \c_\a(E):=\frac{1}{I_\a(E)}.
\end{equation}
We point out that 
this is not the only concept of capacity present in literature. 
Indeed, another one is given by the $2$-capacity of a set $E$, defined by
\begin{equation}\label{pcapacity}
\mathcal{C}_2(E)=\min\left\{\int_{\R^N}|\nabla \varphi|^2:\varphi\in
C^1(\R^N,[0,1]),\,\,\varphi\ge\chi_E\right\}
\end{equation}
where $\chi_A$ is the characteristic function of the set $A$. 
It is possible to prove that, if $E$ is a
compact set, then the minimum in \eqref{pcapacity} is achieved by a function $u$ satisfying
\begin{equation}\label{pcapacityeq}
\left\{ \begin{array}{ll}
  \Delta u=0\qquad &\mbox{on}\,\,\,\R^N\setminus E\\
  u=1\qquad &\mbox{on}\,\,\,E\\
  \lim_{|x|\to+\infty}u(x)=0.
 \end{array}\right.
\end{equation}
It is worth stressing that the $2$-capacity and the $\a$-Riesz capacity 
share several properties, and coincide if $\alpha=2$. 
We refer the reader to \cite[Chapter $8$]{ll} for a discussion of this topic. 

In a series of works (see for instance \cite{cs, Cs,K} and the monography \cite{k}) it has been proved
that the
solutions of
\eqref{pcapacityeq} are level set convex provided $E$ is a convex body, that is,
a compact convex set with non-empty interior.
Moreover, in \cite{borell} (and later in \cite{colsal} in a more general setting and in \cite{colcuo} for the
logarithmic capacity in $2$ dimensions) it has been proved that the
$2$-capacity satisfies a
suitable version of the Brunn-Minkowski inequality: given two convex bodies $K_0$ and $K_1$ in $\R^N$, 
for any $\lambda\in[0,1]$ it holds  
\[
\mathcal{C}_2(\lambda K_1+(1-\lambda)K_0)^{\frac{1}{N-2}}\ge
\lambda\,\mathcal{C}_2(K_1)^{\frac{1}{N-2}}+(1-\lambda)\,\mathcal{C}_2(K_0)^{\frac{1}{N-2}}.
\]
We refer to \cite{sch,gardner} for a comprehensive survey on
the Brunn-Minkowski inequality. 

\smallskip

The main purpose of this paper is to show the analogous
of these results in the fractional setting $\alpha=1$, that is, $s=1/2$ in Problem \eqref{problemaprincipale}. 
More precisely, we shall prove the following result.

\begin{theorem}\label{main}
Let $K\subset\R^N$ be a convex body
and let $u$ be the solution of Problem \eqref{problemaprincipale}
with $s=1/2$. Then
\begin{itemize}
\item[(i)] $u$ is level set
convex, that is, for every $c\in\R$ the set $\{u>c\}$ is convex; 
\item[(ii)] the $1$-Riesz capacity ${\rm Cap}_{1}(K)$ satisfies the following
Brunn-Minkowski inequality: for any couple of convex bodies $K_0$ and $K_1$ 
and for any $\lambda\in[0,1]$ we have
\begin{equation}\label{brunnminkowskiintro}
{\rm Cap}_{1}(\lambda K_1+(1-\lambda)K_0)^{\frac{1}{N-1}}\ge 
\lambda{\rm Cap}_{1}(K_1)^{\frac{1}{N-1}}+(1-\lambda){\rm Cap}_{1}(K_0)^{\frac{1}{N-1}}.
\end{equation}
\end{itemize}
\end{theorem}

The proof of the  Theorem \ref{main} will be given in Section \ref{s2}, 
and relies on the results in \cite{cuosal,colsal} and on the following observation  
due to L. Caffarelli and L. Silvestre.

\begin{proposition}[\cite{cafsil}]\label{caffarellisilvestre}
Let $f:\R^N\to\R$ be a measurable function and let $U:\R^N\times[0,+\infty)$ be the solution of 
 \[
  \Delta_{(x,t)} U(x,t)=0, \,\,\,\text{ on $\R^N\times(0,+\infty)$}\qquad U(x,0)=f(x).
 \]
Then, for any $x\in\R^N$ there holds 
\[
 \lim_{t\to0^+}\partial_t U(x,t)=(-\Delta)^\frac 1 2 f(x).
\]
\end{proposition}

Eventually, in Section \ref{s3} we provide an application of Theorem \ref{main} and we state some open problems.
%

\section{Proof of the main result}\label{s2}

This section is devoted to the proof of Theorem \ref{main}. 

\begin{lemma}\label{lemmas2}
Let $K$ be a compact convex set with positive $2$-capacity
and let $(K_\varepsilon)_{\varepsilon>0}$ be a family of compact convex sets with positive $2$-capacity such that
$K_\e\to K$ in the Hausdorff distance, as $\e\to0$. 
Letting $u_\e$ and $u$ be the capacitary functions of $K_\e$ and $K$ respectively, we have
that $u_\e$ converges uniformly on $\R^N$ to $u$ as $\e\to0$. As a consequence, we have that the sequence
${\mathcal C}_2(K_\e)$ converges to ${\mathcal C}_2(K)$, 
and that the sets $\{u_\e>s\}$ converge to $\{u>s\}$ for any $s>0$, with respect to
the Hausdorff distance.
\end{lemma}

\begin{proof}
We only prove that $u_\e\to u$ uniformly as $\e\to0$ since this  immediately implies the other claims. 
Let $\Omega_\e=K\cup K_\e$. Since $u_\e-u$ is a harmonic function on 
$\R^N\setminus \Omega_\e$, we have that
\begin{equation}\label{formulalemmas2}
\sup_{\R^N\setminus\Omega_\e}|u_\e-u|\le \sup_{\partial \Omega_\e}
|u_\e-u|\le\max\left\{1-\min_{\partial\Omega_\e}u,1-\min_{\partial\Omega_\e}u_\e\right\}.
\end{equation}
Moreover, by Hausdorff convergence, we know that there exists a sequence $(r_\e)_\e$ infinitesimal as $\e\to0$
such that $K_\e\subset K+B_{r_\e}$, where $B(r)$ indicates the ball of radius $r$ centred at the origin. Thus
\begin{equation}\label{formulalemmas3}
\min\left\{
\min_{\partial\Omega_\e}u,\min_{\partial\Omega_\e}u_\e\right\}
\ge\min\left\{\min_{K+B(2r_\e)}u,\min_{K_\e+B(2r_\e)}u_\e\right\}.
\end{equation}
Since the right-hand side of \eqref{formulalemmas3}
converges to $1$ as $\e\to 0$, from \eqref{formulalemmas2}
we obtain 
\[
\lim_{\e\to 0}\sup_{\R^N\setminus\Omega_\e}|u_\e-u|=0,
\]
which brings to the conclusion.
\end{proof}
 
\begin{remark}\rm 
Notice that a compact convex set has positive $2$-capacity if and only if 
its $\mathcal H^{N-1}$-measure is non-zero (see \cite{eg}).
\end{remark}

\begin{proof}[Proof of Theorem \ref{main}]
We start by proving claim $(i)$. Let us consider the problem 
\begin{equation}\label{auxiliaryproblem}
\left\{ \begin{array}{ll}
        -\Delta_{(x,t)} U(x,t)=0 & \mbox{in $\R^N\times(0,\infty)$}\\
        U(x,0)=1 & \mbox{ $x\in K$}\\
        U_t(x,0)=0 &\mbox{in $x\in \R^N\setminus K$}\\
        \lim_{|(x,t)|\to\infty}U(x,t)=0.
        \end{array} \right.
\end{equation}
By Proposition \ref{caffarellisilvestre} we have that
$U(x,0)=u(x)$ for every $x\in\R^N$. 
Notice also that, for any $c\in\R$, we have
\[
 \{u\ge c\}=\{(x,t):U(x,t)\ge c\}\cap\{t=0\}
\]
which entails that $u$ is level set convex, provided that $U$ is level set convex. In order to prove this, we
introduce the problem
\begin{equation}\label{auxiliaryproblem2}
\left\{ \begin{array}{ll}
         \Delta_{(x,t)} V(x,t)=0 & \mbox{in $\R^{N+1}\setminus K$}\\
        V=1 & \mbox{ $x\in K$}\\
        \lim_{|(x,t)|\to\infty}V(x,t)=0
        \end{array} \right.
\end{equation}
whose solution is given by the capacitary function of the set $K$ in $\R^{N+1}$, that is, the function which
achieves the minimum in Problem \eqref{pcapacity}.
%
%

Since $K$ is symmetric with respect to the hyperplane $\{t=0\}$ (where it is contained), it
follows, for instance by applying a suitable version of the P\'olya-Szeg\"o inequality for the Steiner
symmetrization (see for instance
\cite{brock,burchard}), that $V$ is symmetric as well with respect to the same
hyperplane. In particular we have that $\partial_t V(x,0)=0$ for all $x\in\R^N\setminus K$. This implies that
$V(x,t)=U(x,t)$
for every $t\ge0$. To conclude the proof, we are left to check that $V$ is level set convex. To prove this we
recall that the capacitary function of a convex body is level set convex, as proved in \cite{colsal}. Moreover, by
Lemma \ref{lemmas2} applied to the sequence of convex bodies $K_\e=K+B(\e)$ we get that $V$ is level set convex as
well.
This concludes the proof of $(i)$.  

%

To prove $(ii)$ we start by noticing that the $1$-Riesz capacity is a 
$(1-N)$-homogeneous functional, hence 
inequality \eqref{brunnminkowskiintro} can be equivalently stated (see
for instance \cite{borell}) by requiring that, for any couple of
convex sets $K_0$ and $K_1$ and for any $\lambda\in[0,1]$, the inequality 
\begin{equation}\label{brunnminkowski}
 {\rm Cap}_{1}(\lambda K_1+(1-\lambda)K_0)\ge \min\{{\rm Cap}_{1}(K_0),{\rm Cap}_{1}(K_1)\}
\end{equation}
holds true. 

\smallskip

We divide the proof of \eqref{brunnminkowski} into two steps.\\ \\
{\bf Step 1.}\\
We characterize the $1$-Riesz capacity of a convex set $K$ as the behaviour at infinity of the solution of the
following PDE
\[\left\{
\begin{array}{ll}
 (-\Delta)^{1/2} v_K=0\qquad &\mbox{ in $\R^N\setminus K$}\\
 v_K=1 & \mbox{in $K$}\\
 \lim_{|x|\to\infty}|x|^{N-1}v_K(x)={\rm Cap}_{1}(K)
\end{array}\right.
\]
We recall that, if $\mu_K$ is the optimal measure for the minimum problem 
in \eqref{rieszpotentialenergy}, then the function 
\[
 v(x)=\int_{\R^N}\frac{d\mu_K(y)}{|x-y|^{N-1}}
\]
is harmonic on $\R^N\setminus K$ and is constantly equal to $I_{1}(K)$ on $K$ (see for instance
\cite{golnovruf}). Moreover the optimal measure $\mu_K$ is supported on $K$, 
so that $|x|^{N-1}v(x)\to
\mu_K(K)=1$ as $|x|\to\infty$. The claim follows by letting $v_K=v/I_{1}(K)$. \\ \\
{\bf Step 2.}\\
Let $K_\lambda=\lambda K_1+(1-\lambda)K_0$ and 
$v_\lambda=v_{K_\lambda}$. We want to prove that 
\[
v_\lambda(x)\ge \min\{ v_0(x), v_1(x)\}
\]
for any $x\in\R^N$.
To this aim we introduce the auxiliary function
\[
 \widetilde v_\lambda(x)=
\sup\big\{\min\{v_0(x_0),v_1(x_1)\}:x=\lambda x_1+(1-\lambda)x_0\big\},
\]
and we notice that Step $2$ follows if we show that $v_\lambda\ge \v$. 
An equivalent
formulation of this
statement is to require that for any $s>0$ we have
\begin{equation}\label{claim}
\{\v>s\}\subseteq \{v_\lambda>s\}.
\end{equation}
A direct consequence of the definition of $\v$ is that
\[
 \{\v>s\}=\lambda\{v_1>s\}+(1-\lambda)\{v_0>s\}.
\]
For all $\lambda\in [0,1]$, we let  $V_\lambda$ be the harmonic extension 
of $v_\lambda$ on $\R^N\times[0,\infty)$, which solves
\begin{equation}
\left\{ \begin{array}{ll}
        -\Delta_{(x,t)} V_\lambda(x,t)=0 & \mbox{in $\R^N\times(0,\infty)$}\\
        V_\lambda(x,0)=v_\lambda(x) & \mbox{in $\R^N\times\{0\}$}\\ 
        \lim_{|(x,t)|\to\infty}V_\lambda(x,t)=0.
               \end{array} \right.
\end{equation}
Notice that $V_\lambda$ is the
capacitary function of $K_\lambda$ in
$\R^{N+1}$, restricted to $\R^N\times [0,+\infty)$. 
Letting $H=\{(x,t)\in\R^N\times\R:t=0\}$,
for any $\lambda\in[0,1]$ and $s\in\R$ we have 
$$\{V_\lambda>s\}\cap
H=\{v_\lambda>s\}.$$ 
Letting also
\begin{equation}\label{interpolata}
 \widetilde V_\lambda(x,t)=\sup\{\min\{V_0(x_0,t_0),V_1(x_1,t_1)\}:(x,t)=\lambda(x_1,t_1)+(1-\lambda)(x_0,t_0)\},
\end{equation}
as above we have that
\[
\{ \widetilde V_\lambda>s\}=\lambda\{V_1>s\}+(1-\lambda)\{V_0>s\}.
\]
By applying again Lemma \ref{lemmas2} 
to the sequences $K_0^\e=K_0+B(\varepsilon)$ and
$K_1^\e=K_1+B(\varepsilon)$, we get that 
the corresponding capacitary functions, 
denoted respectively as $V_0^\e$ and $V_1^\e$,
converge uniformly to $V_0$ and $V_1$ in $\R^N$, and that 
$\widetilde V_\lambda^\e$, defined as in
\eqref{interpolata}, 
converges uniformly to $\widetilde V_\lambda$ on $\R^{N}\times [0,+\infty)$.
 
Since $\widetilde V_\lambda^\e(x,t)\le V^\e_\lambda(x,t)$
for any $(x,t)\in\R^{N}\times [0,+\infty)$, as shown in
\cite[pages $474-476$]{colsal}, we have that 
$\widetilde V_\lambda(x,t)\le V_\lambda(x,t)$.
As a consequence, we get
\[
\begin{aligned}
\{v_\lambda>s\}&=\{V_\lambda>s\}\cap H\supseteq \{\widetilde V_\lambda>s\}\cap
H=\Big[\lambda\{V_1>s\}+(1-\lambda)\{V_0>s\} \Big]\cap H\\
&\supseteq \lambda\{V_1>s\}\cap H+(1-\lambda)\{V_0>s\}\cap H=\lambda\{v_1>s\}+(1-\lambda)\{v_0>s\}
\end{aligned}
\]
for any $s>0$, which is the claim of {\em Step 2}.

We conclude  by observing that inequality 
\eqref{brunnminkowski} follows immediately,
by putting together {\em Step $1$} and {\em Step $2$}. 
This concludes the proof of $(ii)$, and of the theorem.
\end{proof}

\begin{remark}\rm
 The equality case in the Brunn-Minkowski inequality \eqref{brunnminkowskiintro} is not easy to address
by means of our techniques. 
The problem is not immediate even in the case of the $2$-capacity, for which it has been studied in
\cite{cafjerlie,colsal}.
\end{remark}

\section{Applications and open problems}\label{s3}

In this section we state a corollary of  Theorem \ref{main}. To do this we introduce some tools which arise in
the study of convex bodies. The {\em support function} of a convex body $K\subset\R^N$ is defined on the unit
sphere centred at the origin  $\partial
B(1)$ as
\[
 h_K(\nu)=\sup_{x\in\partial K} \langle x,\nu\rangle.
\]
The {\em mean width} of a convex body $K$ is
\[
 M(K)=\frac{2}{\mathcal{H}^{N-1}(\partial B(1))}\int_{\partial B(1)} h_K(\nu)\,d\mathcal H^{N-1}(\nu).
\]
We refer to \cite{sch} for a complete reference on the subject. 
We observe that, if $N=2$, then $M(K)$
coincides up to a constant with the perimeter $P(K)$ of $K$
(see \cite{bucfralam}).

We denote by $\mathcal K_N$ the set of convex bodies of $\R^N$ and we set
\[
\mathcal K_{N,c}=\{K\in \mathcal K_N,\,M(K)=c\}.
\]
The following result has
been proved in \cite{bucfralam}.
\begin{theorem}\label{bucfralamthm}
 Let $F:\mathcal{K}_N\to[0,\infty)$ be a $q$-homogeneous functional which satisfies the Brunn-Minkowski inequality,
that is, such that $F(K+L)^{1/q}\ge F(K)^{1/q}+F(L)^{1/q}$ for any $K,L\in\mathcal{K}_N$. 
Then the ball is the unique solution of the problem
\begin{equation}
\min_{K\in\mathcal{K}_N} \frac{M(K)}{F^{1/q}(K)}\,.
\end{equation}
\end{theorem}

\noindent
An immediate consequence of Theorem \ref{bucfralamthm}, Theorem \ref{main} and Definition \ref{rieszcapacity} is
the following result.

\begin{corollary}\label{app}
The minimum of $I_{1}$ on the set $\mathcal K_{N,c}$
is achieved by
the ball of measure $c$. In particular, if $N=2$, the ball of radius $r$ solves the 
isoperimetric type problem
\begin{equation}\label{coroeq}
\min_{K\in\mathcal{K}_2, P(K)=2\pi r}I_{1}(K).
\end{equation}
\end{corollary}

Motivated by Theorem \ref{main} and Corollary \ref{app} we conclude the paper with the following conjectures:

\begin{conj}
For any $N\ge 2$ and $\alpha\in (0,N)$, the $\a$-Riesz capacity ${\rm Cap}_{\a}(K)$ satisfies the following
Brunn-Minkowski inequality:\\ for any couple of convex bodies $K_0$ and $K_1$ 
and for any $\lambda\in[0,1]$ we have
\begin{equation}
{\rm Cap}_{\a}(\lambda K_1+(1-\lambda)K_0)^{\frac{1}{N-\a}}\ge 
\lambda{\rm Cap}_{\a}(K_1)^{\frac{1}{N-\a}}+(1-\lambda){\rm Cap}_{\a}(K_0)^{\frac{1}{N-\a}}.
\end{equation}
\end{conj}

\begin{conj}
For any $N\ge 2$ and $\alpha\in (0,N)$, the ball of radius $r$ is the unique solution
of the problem
\begin{equation}\label{conjeq}
\min_{K\in\mathcal{K}_N, P(K)=N\omega_N r^{N-1}}I_{\alpha}(K).
\end{equation}
\end{conj}

%

\section*{Acknowledgements}
The authors wish to thank G. Buttazzo and D. Bucur for useful discussions 
on the subject of this paper.

%

\end{document}